\documentclass[a4paper,10pt]{article}
\usepackage{fullpage}
\usepackage[utf8]{inputenc}
\usepackage{amsmath,amssymb,amsthm}
\usepackage{graphicx}
\usepackage[numbers]{natbib}
\usepackage{url}

\theoremstyle{plain}
\newtheorem{theorem}{Theorem}[section]

\newtheorem{corollary}[theorem]{Corollary}

\newtheorem{conjecture}[theorem]{Conjecture}

\theoremstyle{definition}
\newtheorem{definition}[theorem]{Definition}

\newtheorem*{acknowledgement}{Acknowledgement}

\title{Producing Quality Pseudorandomness with a Generalized Gauss Continued-Fraction Map}
\author{Benjamin V. Holt\\{\it Department of Mathematics, Southwestern Oregon Community College}\\{\tt benjamin.holt@socc.edu}}

\begin{document}

\maketitle

\begin{abstract}
Well-known chaotic maps, such as the logistic and tent maps, have been used to generate cryptographically secure pseudorandomness, yet we know of no efforts which attempt to utilize the Gauss continued-fraction map, a known chaotic map, as a starting point for producing quality pseudorandom output. In this paper, we consider the family of $r$-continued-fraction maps, which generalize the Gauss map, and use them to generate pseudorandom output which outperforms many standard generators, such as the Mersenne Twister, in statistical quality, as ascertained by the use of the Dieharder, PractRand, and TestU01 suites. In this way, we demonstrate the potential viability of these maps as a starting point for novel generators, and provide practical motivation for further study of the properties of both the exact and finite-precision $r$-continued-fraction maps.
\end{abstract}

\maketitle

\section{Introduction}\label{intro}

Monte Carlo simulation, cryptographically secure communications, and machine learning are among the most important applications of pseudorandom number generators \cite{ahmad, antunes,araki,behnia,foreman,naik,yang} (PRNGs). Unlike a sequence of numbers generated by a truly random process, pseudorandom sequences are created by a deterministic algorithm; they are not random at all. That said, ``good'' PRNGs produce sequences of numbers which are difficult to distinguish from truly random sequences. This is to say that sequences produced by a good PRNG should have statistical properties which are at least similar to those of sequences generated by a truly random process.
Chaotic maps have been successfully applied to generating pseudorandom sequences \cite{ahmad,araki, behnia,irfan,naik,yang}. The \textit{logistic map} stands out among these as it has been modified to create secure PRNG algorithms \cite{ahmad,irfan,naik,wang}. Other efforts to modify chaotic maps for the purposes of creating PRNGs include variations on the \textit{tent map} \cite{irfan, yang}. In this paper, we develop a PRNG which is novel in its use of a modification of the lesser-known, but nonetheless chaotic \cite{corless}, Gauss continued-fraction map. The sequences produced by this modified map pass several well-known batteries of statistical randomness tests, demonstrating the Gauss continued-fraction map's promise as a starting point for future development of chaotic PRNGs.

We begin by introducing the {\it Gauss continued-fraction map}, or more simply, the {\it Gauss map}, $G:[0,1)\rightarrow[0,1)$, defined by
\begin{equation}\label{GCF}
 G(x)=\begin{cases}
      \displaystyle \frac{1}{x}\mod 1, & \mbox{ if $x\in (0,1)$}\\
      0 & \mbox{ if $x=0$}.
     \end{cases}
\end{equation}
The Gauss map is well known for its abundance of intriguing properties which touch upon many areas of both pure and applied realms \cite{corless,khinchin,lagarias}. Chief among its desirable properties, for our purposes, are its chaotic behavior \cite{corless} and ergodicity \cite{corless,khinchin}. Ergodic maps have the additional advantage that their long-term behavior is easier to ascertain due to the agreement of time and space averages \cite{behnia}. Although continued fractions have proven useful in generating pseudorandom sequences having low discrepancy \cite{shallit}, and others have made note of the seeming ``randomness'' of partial quotients in simple continued fraction representations \cite{dodge}, we are unaware of any direct attempt to utilize the Gauss map and its properties as a starting point for constructing a usable PRNG \cite{naik}.

We consider the following generalization of $G$: for each positive real number $r$, define the map $T_{r}:[0,1)\rightarrow [0,1)$ by $T_{r}(x)=G(x/r)$. This map is known as the {\it $r$-continued-fraction map} \cite{mehmetaj}, and we shall refer to it as the {\it $r$-CF} map.  The $r$-CF map may be used to produce an $r$-CF expansion of a real number $x$, by which we mean an expression of the form
\begin{equation}\label{rcf_expansion}
x=a_0+\frac{r}{\displaystyle a_1+\frac{r}{\displaystyle a_2+\frac{r}{\ddots}}},
\end{equation}
where $a_0$ can be any integer, and each $a_t$ for $t \geq 1$ is a natural number. A number's $r$-CF expansion is not necessarily unique \cite{greene,mehmetaj}. However, the ``canonical''\cite{greene} method for finding an $r$-CF expansion is analogous to how $G$ is used to find a simple continued fraction expansion: set $x_0=x$, and then generate the sequence $x_{t}=T_r(x_{t-1})$. The partial quotients are then given as $a_0=\lfloor x_0 \rfloor$ and $a_t=\lfloor \frac{r}{x_{t-1}} \rfloor$ for $t \geq 1$. If $x_{t-1}=0$, the process terminates, resulting in a finite expansion. The above $r$-CF expansion is referred to as the {\it maximal} expansion \cite{greene}.
In line with other efforts \cite{anselm,burger,dajani}, an $r$-CF expansion, as displayed in Equation (\ref{rcf_expansion}), will be denoted as $[a_0,a_1,a_2,\ldots]_r$. Greene and Schmieg \cite{greene}, as well as Mehmetaj \cite{mehmetaj}, show that if $r$ is a real number greater than $1$, then every real number $x$ has a valid $r$-CF expansion. Moreover, for any integer $a_0$, and any sequence of positive integers $a_1$, $a_2$,$\ldots$, the limit $\lim_{t\rightarrow \infty}[a_0,a_1,a_2,\ldots,a_t]_r$ exists and is a real number \cite{greene, mehmetaj}. The reader will note that the above is a generalization of $N$-continued fractions, studied by Burger \cite{burger} and others \cite{anselm,dajani,dajani_2,jonge}, which are defined similarly, only that $N$ is a natural number. It is again worth noting that while every real number has at most two simple continued fraction expansions, this is not necessarily the case for $N$-CF and $r$-CF expansions. In fact, any irrational $x>0$ has infinitely many $N$-continued fraction expansions for any fixed $N \in \mathbb{N}\backslash \{1\}$ \cite{anselm,dajani}. Greene and Schmieg \cite{greene} prove the analogous result for $r \geq 2$, while showing that an $r$-CF expansion is unique when $0 < r < 1$. We note here that the maximal $N$-CF expansion is sometimes referred to as the {\it greedy} expansion \cite{dajani}. Finally, for the case of $N$-continued fractions, some classical properties hold. For instance, it is still the case that every quadratic irrational has an eventually periodic expansion \cite{anselm,greene}. On the other hand, these properties can break down when considering the $r$-CF expansion, even when $r$ is rational \cite{greene}.

Here, we present a simple method for producing what appears to be quality pseudorandom output which utilizes the $r$-CF map. Section \ref{rcf} summarizes features of the $r$-CF map which are desirable for producing pseudorandom output. In Section \ref{rcfprng}, we present a simple algorithm for generating pseudorandom output which utilizes these features. In Section \ref{implementation}, we detail how we implemented the algorithm described in Section \ref{rcfprng}. In Section \ref{results_and_discussion}, we present and discuss the results of our simulations, which, among other things, assesses the statistical quality of the output. The forward- and reverse-bit output routinely passes all tests in the Dieharder \cite{brown} suite, the PractRand \cite{practrand} suite (up to $8$ terabytes of output with further testing required), as well as all tests in the BigCrush battery, which is included the TestU01 suite \cite{testu01}. In Section \ref{future_directions}, we summarize our efforts, discuss the limitations of the methods proposed here, outline possible future efforts for addressing and improving upon these shortcomings, and present possible variations on the ideas presented in the body of the paper.

\section{Key Features of the $r$-CF Map}\label{rcf}

In this section, we highlight the most important features of the $r$-CF map for our purposes.

\subsection{Ergodicity}\label{rcf_ergodic}

\begin{definition}
For each real number $r \geq 1$, define the {\it $r$-Gauss-Kuzmin} measure $\mu_{r}$ by
\[
\mu_{r}(B)=\frac{1}{\log\left(1+\frac{1}{r}\right)}\int_B \frac{\textup{d}\lambda}{r+x},
\]
where $\lambda$ is the Lebesgue measure.
\end{definition}

Clearly, the above generalizes the Gauss-Kuzmin measure. The following is a result by Dajani, Kraaikamp, and Wekken \cite{dajani}.

\begin{theorem}
When $r$ is a natural number, the measure $\mu_{r}$ is $T_{r}$-invariant and ergodic.
\end{theorem}

On the other hand, for non-integer values of $r$, the $r$-CF map no longer preserves $\mu_r$. Indeed, for $0\leq a <b<1$,
\begin{align*}
  \mu_{r}(T_{r}^{-1}[a,b])&=\mu_{r}\left(\bigcup_{k=\lceil r-a \rceil}^{\infty}\left[\frac{r}{k+b},\frac{r}{k+a}\right]\right)\\
 &=\frac{1}{\log\left(1+\frac{1}{r}\right)}\sum_{k=\lceil r-a \rceil}^{\infty}\int_{{r}/(k+b)}^{{r}/(k+a)} \frac{\textup{d}x}{r+x}\\
 &=\frac{\log(\lceil r-a \rceil+b)-\log(\lceil r-a \rceil+a)}{\log\left(1+\frac{1}{r}\right)}\\
 & \neq \frac{\log(r+b)-\log(r+a)}{\log\left(1+\frac{1}{r}\right)}\\
 &=\mu_{r}([a,b]).
 \end{align*}

This is, perhaps, not surprising considering that some classical properties of continued fraction expansions fail to hold when $r$ is no longer a natural number. However, we do see that for sufficiently large values of $r$, the approximation $\mu_{r}([a,b])\approx \mu_{r}(T_{r}^{-1}[a,b])$ is valid. In fact, we can say a little more.

\begin{corollary}\label{approx_lebesgue}
For any Borel set $B$ of $[0,1)$, we have
$
\lim_{r \rightarrow \infty} \mu_{r}(B)=\lim_{r \rightarrow \infty}\/\mu_{r} \left(T_{r}^{-1}B\right)=\lambda(B).
$
\end{corollary}
\begin{proof}
It suffices to show that the result holds on any subinterval of $[0,1)$. Let $L(r) =\frac{\log(r+1+b)-\log(r+1+a)}{\log\left(1+\frac{1}{r}\right)}$ and $U(r)=\frac{\log(r-1+b)-\log(r-1+a)}{\log\left(1+\frac{1}{r}\right)}$. With the previous calculation in mind, we observe that $L(r) < \mu_{r}([a,b]) < U(r)$. Since $r-1 < \lceil r-a \rceil < r+1$, we also have $L(r) < \mu_{r}(T_{r}^{-1}[a,b]) < U(r)$. It then follows that $\lim_{r \rightarrow \infty} \mu_{r}([a,b])=\lim_{r \rightarrow \infty} \mu_{r}(T_{r}^{-1}[a,b])=b-a$.
\end{proof}

Applying Birkhoff's ergodic theorem, together with Corollary \ref{approx_lebesgue}, gives us that sequences $\{x_t\}$ defined by $x_{t+1}=T_r(x_t)=G(x_t/r)$, where $r$ is a sufficiently large natural number and $x_0$ is chosen randomly from $[0,1]$, are approximately uniformly distributed modulo $1$.

As for the ergodicity of the map for non-integer values of $r$, the following is a theorem of Mehmetaj \cite{mehmetaj}.

\begin{theorem}\label{mehmetaj}
Suppose $\triangle(d)=\left(\frac{r}{\lfloor r \rfloor+1},1\right)$ when $d=\lfloor r \rfloor$, and $\triangle(d)=\left(\frac{r}{d+1},\frac{r}{d}\right)$ when $d$ is a natural number. Let $\{\triangle(d) \subset [0,1) \mid d \geq \lfloor r \rfloor \}$ be a countable partition of $[0,1)$. Then $T_r$ admits an ergodic, absolutely continuous, invariant measure equivalent to the Lebesgue measure.
\end{theorem}

\subsection{Chaotic Behavior}\label{rcf_chaos}

Neither Mehmetaj \cite{mehmetaj} nor Greene and Schmieg \cite{greene} consider possible chaotic features of the $r$-CF map. In this section we present numerical evidence of an everywhere-positive Lyapunov exponent, $\lambda$, for $r \geq 1$. Using the well-known definition for the case of a discrete map,
\[
\lambda=\lim_{N \rightarrow \infty} \frac{1}{N}\sum_{t=0}^{N-1} \log |T_r'(x_t)|,
\]
we estimated $\lambda$ for a pseudorandomly chosen value of $r$ from a specified interval by computing the orbits of a pseudorandomly chosen point $x_0$ drawn from $(0,1)$, and then estimating $\lambda$ from the orbit, taking $N=10^5$.

\begin{figure*}
\centering
\begin{tabular}{cc}
 \includegraphics[scale=0.6]{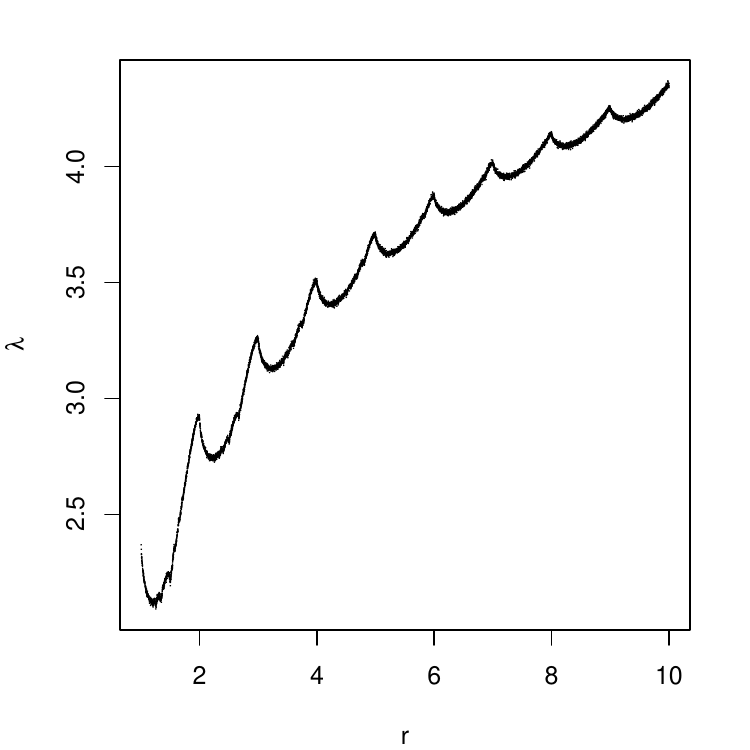} & \includegraphics[scale=0.6]{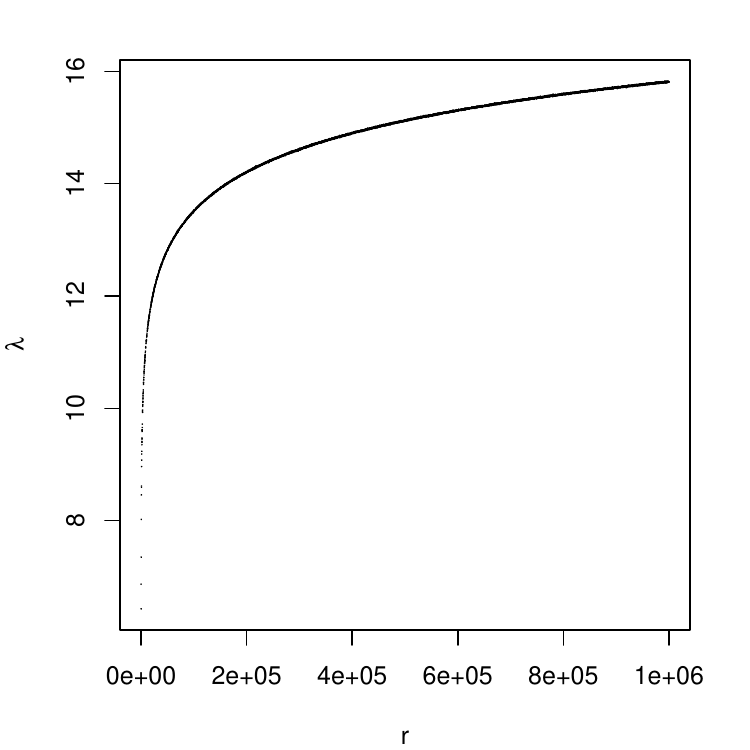} \\
 (a) & (b)
\end{tabular}
\caption[figure]{The estimated Lyapunov exponent, $\lambda$, of $T_r$ as $r$ varies. Each graph was generated by pseudorandomly sampling $10000$ values of $r$ from the intervals $[1,10]$ (a) and $[1,10^6]$ (b), and then for each value of $r$, estimating $\lambda$ from an orbit of a pseudorandomly chosen point in $(0,1)$.}\label{lyap}
\end{figure*}

Figure \ref{lyap} suggests, in a global sense, that the $r$-CF map becomes more chaotic as $r$ increases. Also, given the logarithmic shape of $\lambda$ versus $r$, and given the numerical evidence, Figure \ref{lyap} suggests that $|2+\log(r)-\lambda| \rightarrow 0$ as $r \rightarrow \infty$. Later, in Section \ref{future_directions}, we will formally state a conjecture for a two-parameter generalization of the $r$-CF map. In this way, we take this as evidence that the $r$-CF map is sensitive to initial conditions. The figure also suggests that the curve exhibits finer-scale features which might be fractal in nature. Figure \ref{lyap_alpha}(a) in Section \ref{future_directions} also supports this suspicion. Finally, in a more local sense, it appears that the map is maximally chaotic near integer values as indicated by the cusp-like features of the curve.

With regard to the above, it must be emphasized that the finite-precision arithmetic may change the simulated dynamics of the $r$-CF map, that is, we cannot be entirely certain that simulated orbits of $T_r$ bear any resemblance to those of the exact map \cite{corless}. We shall say more about this in Section \ref{future_directions}.

Beyond evidence of an almost-everywhere positive Lyapunov exponent, the $r$-CF map also plainly exhibits other classical chaotic features. Consider a real number in $(0,1)$ with $r$-CF expansion $[0,a_1,a_2,\ldots]_r$, where $a_t \geq r$ for all $t \geq 1$. It is not difficult to see that $T_r([0,a_1,a_2,\ldots]_r)=[0,a_2,a_3,\ldots]_r$, which is to say that the $r$-CF map acts as a shift map on certain types of $r$-CF expansions. Thus, by virtue of these representations, we may manufacture periodic points of any period, of which, there are countably many. On the other hand, by this same reasoning, there are uncountably many aperiodic orbits.

Other properties to consider are mixing properties. The Gauss map is known to be exponentially mixing \cite{policott}. That said, there seems to be little else known about mixing properties of the $r$-CF map. Given the success of the methods presented here, we suspect that the $r$-CF map inherits some degree of mixing from the Gauss map. This is one area for further investigation.

\subsection{Probabilistic Considerations}\label{rcf_prob_stat}

The following is another suggestive observation in line with the above which helps to inform the reasoning behind the methods presented in Section \ref{rcfprng}.

\begin{theorem}\label{approx_uniform}
 Let $X$ be a uniform random variable on $[0,1]$, and let $r>1$ be a real number. Then, the random variable $T_{r}(X)$ converges in distribution to a uniform random variable on $[0,1]$ as $r \rightarrow \infty$.
\end{theorem}

\begin{proof}
We will show that $\displaystyle \lim_{{r} \rightarrow \infty}F_{T_{r}(X)}(x)=x$. Since
\[
\begin{array}{rl}
F_{T_{r}(X)}(x)
&\displaystyle =\mathbb{P}(T_{r}(X) \leq x)\\
&\displaystyle =\mathbb{P}\left(\frac{X}{r}\in \bigcup_{k=\lceil r-a \rceil}^{\infty}\left[\frac{1}{k+x},\frac{1}{k}\right]\right)\\
&\displaystyle =\mathbb{P}\left(X\in \bigcup_{k=\lceil r-a \rceil}^{\infty}\left[\frac{r}{k+x},\frac{r}{k}\right]\right)\\
&\displaystyle =\sum_{k=\lceil r-a \rceil}^{\infty}\left(\frac{r}{k}-\frac{r}{k+x}\right)\\
&\displaystyle=x\sum_{k=\lceil r-a \rceil}^{\infty}\frac{r}{k(k+x)},\\
\end{array}
\]
it only remains to be shown that $\displaystyle \lim_{r \rightarrow \infty}\sum_{k=\lceil r-a \rceil}^{\infty}\frac{r}{k(k+x)}=1$. Comparing series, we see that
\begin{align*}
\frac{r}{\lceil r-a \rceil}
&=\sum_{k=\lceil r-a \rceil}^{\infty}\frac{r}{k(k+1)}\\
&<\sum_{k=\lceil r-a \rceil}^{\infty}\frac{r}{k(k+x)}\\
&<\sum_{k=\lceil r-a \rceil}^{\infty}\frac{r}{k(k-1)}\\
&=\frac{r}{\lceil r-a \rceil-1},
\end{align*}
which establishes the result.
\end{proof}

The following is an immediate corollary.

\begin{corollary}\label{iterates}
 Let $X$ be a uniform random variable on $[0,1]$,  and let $r>1$ be a real number, and let $t$ be a natural number. Then, the random variable $T_{r}^t(X)$ converges in distribution to a uniform random variable on $[0,1]$ as $r \rightarrow \infty$.
\end{corollary}

\subsection{A Physical Analogy}

To gain a more visceral understanding of the dynamics of the $r$-CF map, we extend a physical analogy of the Gauss map made by  Corless, Frank, and Monroe \cite{corless}: suppose a bead is on a circular hoop located at an initial point $x_0$. Identifying the endpoints of $[0,1]$, we let the unit circle, $\mathbb{R}/\mathbb{Z}$, model this hoop. We may think of the Gauss map as taking this bead to a new location on the hoop, sliding it past the origin $a_1=\lfloor 1/x_0 \rfloor$ times (a ``winding number'' of sorts), and ending up at the position $x_1=G(x_0)$. Where the bead lands next on the hoop depends heavily upon the previous point $x_0$. The closer $x_0$ is to $0$, the stronger the next ``push'' which sends the bead around the hoop, and the larger the winding number. Continuing this process, we may iteratively push the bead around the hoop to its next point, $x_{t+1}=G(x_t)$, where the strength of each push depends upon the previous point, sending the bead past the origin $a_{t+1}=\lfloor 1/x_t \rfloor$ times.  The successive bead positions represent the tails of the simple continued fraction expansion of the real number $x=a_0+x_0$, and the partial quotients, $a_t$, are the winding number of each push. Since partial quotients are generally unbounded, we see that the map can send the bead around the hoop as many times as we like by making the initial point sufficiently close to $0$. If instead we use the $r$-CF map to find the maximal $r$-CF expansion of $x$, the larger $r$ becomes, the less information the initial point yields about the possible location of the next. For example, having no information other than the initial point being located in $[1/2,1]$, the Gauss map is more likely to return a value in the interval $[0,1/2]$ than $[1/2,1]$, as seen in Figure \ref{maps}(a). That is, if the bead's initial location is somewhere in the interval $[0,1/2]$, the next push will be so weak that it will more than likely not slide past the point $1/2$ on the hoop. In this way, the analogy for $r=1$ serves to hone our intuition for why relative frequencies of partial quotients are described by the Gauss-Kuzmin distribution in the case of simple continued fractions. On the other hand, we immediately see that by allowing $r$ to take on values greater than $1$, say $r=10$ (see Figure \ref{maps}(b)), the strength of the next push has increased substantially enough that the certainty of where the bead will land next is greatly reduced, even for a relatively small value of $r$. In this way, the extended analogy informs our intuition for why Theorem \ref{approx_uniform} and its corollary are true, as well as the reasoning behind using the $r$-CF map as a starting point for a PRNG.

\begin{figure*}
\centering
\begin{tabular}{cc}
 \includegraphics[scale=0.85]{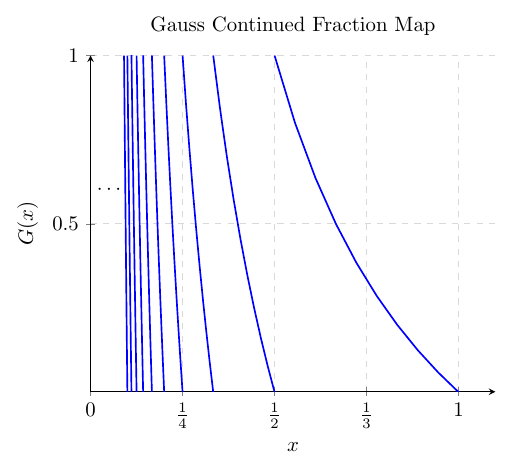} & \includegraphics[scale=0.85]{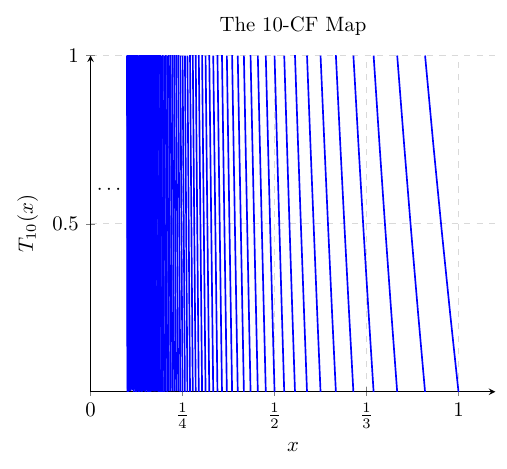} \\
 (a) & (b)
\end{tabular}
\caption[figure]{Graphs of the Gauss map (a) and $10$-CF map (b). The graphs are incomplete, as indicated by the notation ``$\cdots$'', since each one contains infinitely many branches.}\label{maps}
 \end{figure*}

\section{Pseudorandomness with the $r$-CF Map}\label{rcfprng}

From Section \ref{rcf}, we see that sequences $\{x_t\}$ defined by $x_{t+1}=G(x_t/r)$, for almost any initial point, are approximately uniform modulo $1$ for sufficiently large natural-number values of $r$. The result appears to apply to sufficiently large real-number values of $r$ as well, which is not surprising given the statement of Theorem \ref{mehmetaj}. Choosing $r$ to be a suitably large value approximating an irrational quantity, that is, assigning to $r$ a pseudorandomly chosen floating-point value between $10^3$ and $10^4$ (we will say more about this choice in Sections \ref{implementation} and \ref{future_directions}), and choosing a pseudorandom seed $x_0$ in $(0,1)$, these sequences pass most of the Dieharder tests with default parameters. The tests which they consistently fail are the {\tt marsaglia\_tsang\_gcd} and {\tt dab\_bytedistrib} tests. They also see several weak and failed lagged-sum test results. These sequences also routinely pass SmallCrush. We suspect that these sequences do not pass all tests in Crush and BigCrush.

Curiously, for certain Dieharder tests, the quality of the output appears to be highly sensitive to the choice of $r$. That is, some tests either routinely yield a pass result or fail result, even for very similar values of $r$. For example, on the machine we used for testing (see Section \ref{implementation}), creatively choosing an $r$ which approximates $\sqrt{12345678}$, the output consistently passes the {\tt dab\_bytedistrib} test, whereas, a choice of $\sqrt{12345679}$ yields consistent failure. Although these sequences show some promise in that they pass many of the Dieharder tests for certain choices of $r$, they still lack other statistical features which make for quality pseudorandom output. For these values of $r$, they also pass nearly all SmallCrush tests with the exception of the {\tt MaxOft} test. To this point, regardless of how we choose $r$, these sequences resoundingly fail in PractRand at or before $1$ gigabyte of output. Most typically failure occurs between $64$ to $256$ megabytes. The most pronounced failures tend to occur in the lower-order bits.

To improve the results, we manufacture another source of variation by allowing $r$ to vary with each iteration. Let $\{R_t\}$ be a sequence of random variables, where each $R_t$ is independently and uniformly distributed on $[A,B]$ with $0<A<B$. Also, let $X$ be uniformly distributed on $[0,1]$. Assuming independence of $X$ and $R_t$ for each $t$, and assuming $A$ and $B$ to be sufficiently large, Corollary \ref{iterates} gives us a family of joint random variables, $T_{R_t}^t(X)$, indexed by $t$, each with approximate cumulative density $F_{X,R_t}(x,y)=x\frac{y-A}{B-A}$. The independence of each $R_t$ guarantees the independence of each $T_{R_t}^t(X)$. Then, for $x_0$ chosen randomly from $(0,1)$, and a sample path $\{r_t\}$ of $\{R_t\}$, we may say that, by an application of the strong law of large numbers and Weyl's criterion, the sequence $\{T_{r_t}^t(x_0)\}$ is approximately uniformly distributed modulo $1$.

If instead $\{r_t\}$ is a pseudorandom sequence, simulating a uniform distribution on $[A,B]$, then, by the above arguments, the sequence $\{T_{r_t}^t(x_0)\}$, which we may calculate as $x_{t+1}=G(x_t/r_t)$, should also simulate an approximately uniform distribution on $[0,1]$. Then, the question becomes how to simulate $r_t$. Here, we introduce another source of pseudorandomness: define a second sequence $\{y_t\}$, computed as $y_{t+1}=G(y_t/s_t)$, where $\{s_t\}$ is also uniformly distributed on the interval $[A,B]$. Since we can expect $\{y_t\}$ to be approximately uniform on $[0,1]$, we may define $\{r_t\}$ by the update rule $r_{t}=A+(B-A)y_{t}$, which we can expect to be uniformly distributed on $[A,B]$. In turn, as argued above, we may still expect that $\{x_t\}$ is approximately uniformly distributed on $[0,1]$. This enables us to then simulate $\{s_t\}$ with the update rule $s_{t}=A+(B-A)x_{t+1}.$ That is, given a seed $(x_0,y_0)$, where $x_0$ and $y_0$ are randomly chosen from $(0,1)$, we have a vector, $(x_t,y_t)$, which may be updated according to the rule
\begin{align*}
 r_t&=A+(B-A)y_t, \\
 x_{t+1}&=G(x_t/r_t), \\
 s_t&=A+(B-A)x_{t+1},\\
 y_{t+1}&=G(y_t/s_t).
\end{align*}
Implementing the above, and choosing $(x_0,y_0)$ pseudorandomly, preliminary testing of the output stream, $\lfloor x_0 \times 2^{32} \rfloor, \lfloor y_0 \times 2^{32}\rfloor, \lfloor x_1 \times 2^{32} \rfloor, \lfloor y_1 \times 2^{32}\rfloor,\ldots$, showed substantially improved results: both the forward- and reverse-bit versions passed all tests in PractRand to one terabyte (without further testing) and consistently passed all tests in SmallCrush. Here, we note that similar two-dimensional systems based on the logistic map, rather than the $r$-CF map, have been applied successfully to creating PRNGs suitable for cryptographic uses \cite{ahmad}.

To simplify both notation and implementation, we restate the above equations in terms of assignment operators:
\begin{align*}
 r&:=A+(B-A)y, \\
 x&:=G(x/r), \\
 s&:=A+(B-A)x,\\
 y&:=G(y/s).
\end{align*}
With an eye toward more robust performance, the above algorithm is easily generalized to an arbitrary number of dimensions, $n$, without any  compromise to its speed (see Table \ref{benchmark_table} in Section \ref{speed}). Let $(x_0,x_1,\ldots,x_{n-1})$ be a vector (array) of values in $(0,1)$. At each step, we update $x_j$ according to the assignment rule
\begin{align}
 r_j&:=A+(B-A)x_{\rho(j+1)},\label{x_update_0}\\
 x_j&:=G(x_j/r_j),\label{x_update_1}\\
 j&:=\rho(j+1),\label{x_update_2}
 \end{align}
where $\rho$ gives the least non-negative residue modulo $n$. The updated value $x_j$ is the output at each step.

With the reasoning of the above understood, we now write Equations (\ref{x_update_0}) and (\ref{x_update_1}) into a single update rule,
\begin{equation}\label{x_update_3}
 x_j:=G\left(\frac{x_j}{A+(B-A)x_{\rho(j+1)}}\right).
\end{equation}

It will be notationally convenient to reintroduce a ``time'' step, $t$, incremented as $t:=t+1$, among the list of variables contained in Equations (\ref{x_update_0}), (\ref{x_update_1}), and (\ref{x_update_2}). We will denote the $j^{th}$ component of the state vector, $(x_0,x_1,\ldots,x_{n-1})$, at time $t$ as $x_j^{(t)}$. We may then re-express the update rule in terms of $t$: at each time step, $t$, the $\rho(t)^{th}$ component of the state vector, $(x_0^{(t)}, x_1^{(t)}, \ldots, x_{n-1}^{(t)})$, is updated by the equality statement
\begin{equation}\label{x_update}
 x_{\rho(t)}^{(t+1)}=G\left(\frac{x_{\rho(t)}^{(t)}}{A+(B-A)x_{\rho(t+1)}^{(t)}}\right)
\end{equation}
to obtain the next state, $(x_0^{(t+1)}, x_1^{(t+1)}, \ldots, x_{n-1}^{(t+1)})$. The output of the generator at time $t$ is $x_{\rho(t)}^{(t+1)}$.

We will refer to the generator described above as the {\it $r$-CF generator}.

\section{Implementation}\label{implementation}

The $r$-CF generator, as given by Equation (\ref{x_update_3}), was implemented and tested  in C++. All code was compiled using the \texttt{g++} compiler in Linux. The hardware used for building, compiling, and running all software, including Dieharder, PractRand (v0.96), TestU01, and Google Benchmark software \cite{google}, was a Dell XPS $13$ laptop with an Intel Core i7-8565U, $8$th generation processor. The generator function was implemented as follows:
\begin{verbatim}
double G(long double x)
{
     if(x>0 && x<1)
          x=1/x-floor(1/x);
     else
          x=double(rand())/double(RAND_MAX);
     return x;
}
const long long BIG_NUM=pow(2,32);
const double A=1000;
const double B=10000;
const int n=1000;
static int j=0;
static double x[n];
uint32_t rcf(void)
{
     x[j]=G(x[j]/(A+(B-A)*x[(j+1)%n]));
     uint32_t next_num = floor(x[j]*BIG_NUM);
     j=(j+1)%n;
     return next_num;
}
\end{verbatim}

The generator was seeded by initializing the array {\tt x[n]} with pseudorandomly chosen values from $(0,1)$, calculated as {\tt double(rand())/double(RAND\_MAX)}, where \texttt{rand()} is the native PRNG in the standard C library, seeded by the machine clock before each run. As seen in the above code fragment, for all statistical randomness testing, sequences were produced with the parameters $A=1000$, $B=10000$, and $n=1000$. Also, to simulate the dynamics of $G$ as closely as possible, its argument was declared as a {\tt long double} type. Declaring any lower-bit type resulted in poor performance when running statistical tests. All other floating-point types were declared as {\tt double}. To account for possible overflow issues in the finite-precision arithmetic, the definition of $G$, expressed by Equation (\ref{GCF}), was modified to reseed the generator in the case that $G$ returns a value which is not in $(0,1)$. We suspect that this modification may be the reason why allowing $A$ and $B$ to take on very large values tends to produce lower quality output. This explains the choice of parameters $A$ and $B$ above: these values circumvent overflow issues while still producing values, $r_j$, which ensure that the map $G(x/r_j)$ produces approximately uniform output. A fuller discussion of the limitations of finite-precision arithmetic when simulating the $r$-CF map is given in Section \ref{future_directions}.

With the above, we may test the output in Dieharder, PractRand, and TestU01 by producing a stream of $32$-bit integers, $\left \lfloor x_{\rho(t)}^{(t+1)} \times 2^{32} \right \rfloor$. The output was piped directly into the Dieharder and PractRand suites in raw binary form. The $32$-bit integer output was also tested in the TestU01 environment. We note that the above generator may be easily modified to a $64$-bit version. For convenience, however, in utilizing certain test environments, we only implemented and tested the $32$-bit version.

\section{Results and Discussion}\label{results_and_discussion}

\subsection{Graphical Analysis}

Checking the decimal output for obvious flaws, anomalies, biases, or any other artifacts that would suggest compromised performance, we examined the usual graphical summaries of the output. Figure \ref{histogram} displays histograms of $10^5$ consecutive outputs for a typical run produced by the method described in Section \ref{rcfprng}. Figure \ref{scatter} plots ordered pairs of successive outputs, $(x_{\rho(t)}^{(t+1)},x_{\rho(t+1)}^{(t+2)})$, for these same $10^5$ consecutive outputs. The autocorrelation graph for the vector containing these $10^5$ values is shown in Figure \ref{autocorr}. In all, what we observe does not differ in any noticeable way from what we would expect from truly random noise.

\begin{figure*}
\centering
\begin{tabular}{cc}
 \includegraphics[scale=0.6]{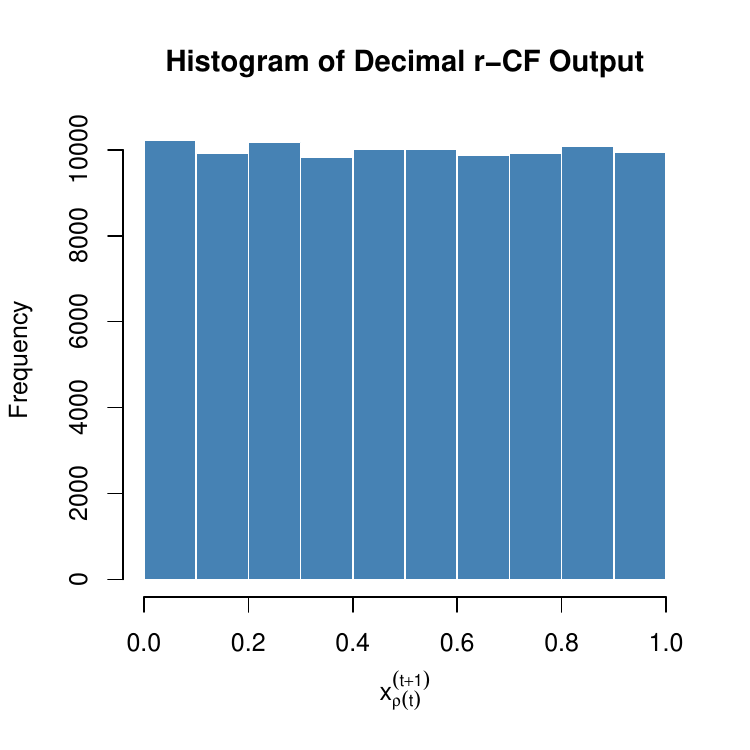} & \includegraphics[scale=0.6]{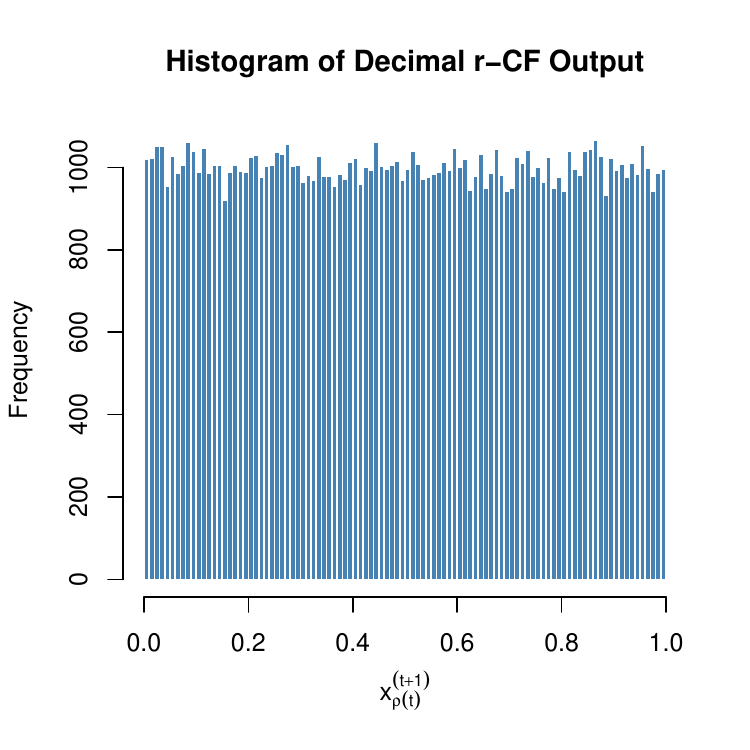} \\
 (a) & (b)
\end{tabular}
\caption[figure]{Frequency histograms of $10^5$ consecutive outputs, $x_{\rho(t)}^{(t+1)}$, of a typical run for $10^5\leq t < 2\times 10^5$ with a bin width of $0.1$ (a) and $0.01$ (b).}\label{histogram}
\end{figure*}

\begin{figure}
\centering
 \includegraphics[scale=0.6]{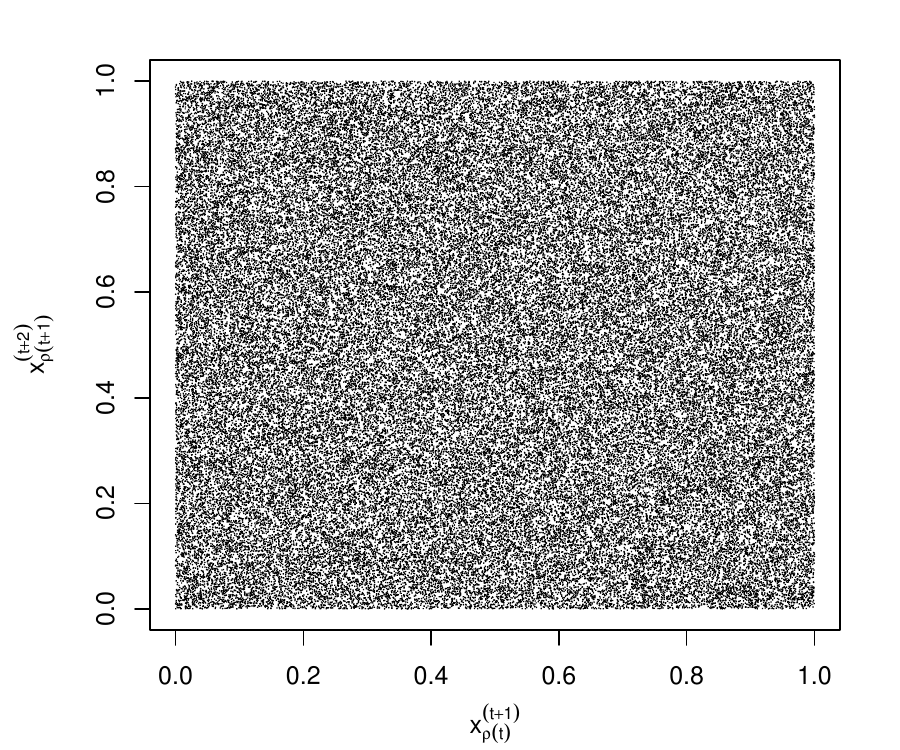}
\caption[figure]{A plot of $10^5$ consecutive outputs, $(x_{\rho(t)}^{(t+1)},x_{\rho(t+1)}^{(t+2)})$, where $10^5\leq t < 2\times 10^5$.} \label{scatter}
 \end{figure}

\begin{figure}
\centering
 \includegraphics[scale=0.7]{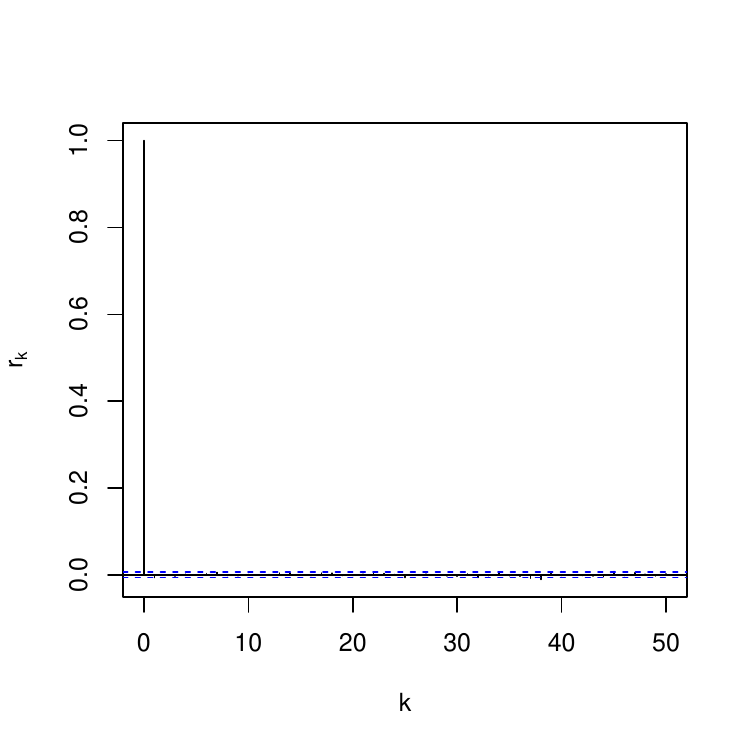}
\caption[figure]{A plot of the the autocorrelation, $r_k$, against the lag, $k$, of a vector of $10^5$ consecutive outputs, $x_{\rho(t)}^{(t+1)}$, where $10^5\leq t < 2\times 10^5$.}\label{autocorr}
\end{figure}

\subsection{Statistical Testing}

We tested the $r$-CF generator using three testing suites, Dieharder, PractRand, and TestU01, using default parameters in each suite. Here, the reader will recall that the generator was re-seeded by the machine clock at the beginning of each run, as described in Section \ref{implementation}. Other than the expected weak or suspicious results, the output consistently passes every test in each of these suites.

Over five runs of each, the forward- and reverse-bit output easily passes Dieharder, and we encountered no failed or consistently weak results in any of its tests.

In PractRand, the output from four runs of both the forward- and reverse-bit generators passed all tests up to $8$ terabytes of output. We have not tested beyond this due to prohibitively long run times. More testing will need to be done with this software, preferably to at least $32$ terabytes.

The most extensive testing was carried out in TestU01, and consisted of twenty consecutive runs of the BigCrush battery: ten testing the forward-bit output, and ten testing the reverse-bit output. The results are summarized in Tables \ref{bigcrush} and \ref{bigcrush_rev}, respectively. In its end-of-run summary, TestU01 flags any of its test statistics if its corresponding $p$-value does not belong to $[0.001,0.999]$. As expected, over repeated applications of the BigCrush battery, we observed some flagged results, as seen in Tables \ref{bigcrush} and \ref{bigcrush_rev}.

\begin{table}\centering
\begin{tabular}{llll}
Run & Test & Summary results of BigCrush & $p$-value\\
\hline
1  & -- & All tests were passed    &\\
2  & -- & All tests were passed    &\\
3  & 47 & MaxOft AD, t = 16 & 0.9995 \\
4  & -- & All tests were passed    &\\
5  & -- & All tests were passed &\\
6  & 72 & Savir2 & 4.2e-4 \\
7  & -- & All tests were passed    &\\
8  &  9 & CollisionOver, t = 14    &       0.9993\\
   & 101&  Run of bits, r = 0       &       1.9e-5\\
9  & -- & All tests were passed    &\\
10 & 32 & CouponCollector, r = 20   & 8.2e-4\\
\end{tabular}
\caption{The results of ten consecutive forward-bit runs of the BigCrush battery, which include tests yielding a $p$-value outside of the interval $[0.001,0.999]$ in the end-of-run summary. Each run used its own pseudorandomly chosen seed as described in Section \ref{implementation}.}\label{bigcrush}
\end{table}

\begin{table}\centering
\begin{tabular}{llll}
Run & Test & Summary results of BigCrush & $p$-value\\
\hline
1   &  -- & All tests were passed    &\\
2   & 24 & ClosePairs NP, t = 9 & 0.9995  \\
3   & -- & All tests were passed    &\\
4   & -- & All tests were passed    &\\
5   & 75 &  RandomWalk1 R (L=50, r=25) & 0.9997\\
6   & 7  & CollisionOver, t = 7 &  0.9998 \\
7   & -- & All tests were passed    &\\
8   & -- & All tests were passed    &\\
9   & -- & All tests were passed    &\\
10  & -- & All tests were passed    &\\
\end{tabular}
\caption{The results of ten consecutive reverse-bit runs of the BigCrush battery, which include tests yielding a $p$-value outside of the interval $[0.001,0.999]$ in the end-of-run summary. Each run used its own pseudorandomly chosen seed as described in Section \ref{implementation}.}
 \label{bigcrush_rev}
\end{table}

According to the protocol of L'Ecuyer and Simard \cite{testu01}, a generator exhibits {\it clear failure} when the $p$-value of a test falls outside of the interval $[10^{-10},1-10^{-10}]$. None of the flagged results in Tables \ref{bigcrush} and \ref{bigcrush_rev} represent clear failures according to this standard. In fact, most suspect $p$-values are borderline (near the endpoints of $[0.001,0.999]$). Moreover, we did not observe any consistently suspicious behavior in any of the $106$ tests (consisting of $160$ test statistics) in the BigCrush battery over both the forward- and reverse-bit runs. L'Ecuyer and Simard \cite{testu01} state the following:
\begin{quote}
If the $p$-value is suspicious but does not clearly indicate rejection ($p= 0.002$, for example), then the test can be replicated `independently' with disjoint output sequences from the same generator until either failure becomes obvious or suspicion disappears.
\end{quote}
Since the generator was re-seeded at the start of each run, we may consider each run to be independent of any other in a pseudorandom sense. This is to say that, for both the forward- and reverse-bit generators, each test was ``independently'' replicated $10$ times with no repeatedly suspicious results in any single test as seen in Tables \ref{bigcrush} and \ref{bigcrush_rev}. Thus, in line with the recommendation of L'Ecuyer and Simard \cite{testu01}, we conclude that the output of the $r$-CF generator does not exhibit any of the systemic flaws in pseudorandom output which the BigCrush battery is intended to detect.

Comparing the above results to those of a  generator in widespread use, we point out that the Mersenne Twister \cite{mersenne} routinely fails certain tests in TestU01 \cite{testu01,oneill}, most notably those involving linear complexity. On the other hand, as exhibited by Tables \ref{bigcrush} and \ref{bigcrush_rev},
the $r$-CF generator routinely passes all tests in BigCrush. Also, since we have not observed a single failure in TestU01, the $r$-CF generator appears to hold its own quite well when compared to those tested by L'Ecuyer and Simard \cite{testu01}. We also note that the Mersenne Twister begins failing PractRand at around $256$ to $512$ gigabytes of output \cite{practrand,oneill_2}, while the $r$-CF generator over four consecutive forward- and reverse-bit runs has passed all tests in PractRand up to $8$ terabytes. To this we add that in all of our testing in PractRand, which involved scores of runs examining various volumes of output spanning $32$ gigabytes to one terabyte, we have not observed a single failure, with the exception of those runs for which a deliberately poor choice of seed was made (see Section \ref{seed}).

In total, the above results provide solid evidence that the $r$-CF generator produces quality pseudorandom output. Of course, while these results are promising, we caution that more testing still needs to be done. In particular, we recommend testing larger volumes of output in PractRand, at least up to $32$ terabytes. Also, it would be desirable to use testing software beside that which we have used here. For example, the documentation of PractRand \cite{practrand} recommends the use of gjrand \cite{gjrand}, which we did not use in our evaluation of the $r$-CF generator. Another example of testing software worthy of note is the emerging SmokeRand \cite{smokerand} suite, which incorporates features from existing suites while expanding flaw-detecting capabilities. Future efforts might consider applying the above testing software. We also note that an assessment of the performance of a $64$-bit version of the $r$-CF generator has yet to be carried out. Finally, it must be emphasized that while statistical testing may be useful for singling out known flaws in many PRNGs \cite{practrand,foreman,testu01}, much analysis remains to be done on both the exact and finite-precision $r$-CF maps and their interactions described by Equations (\ref{x_update_3}) and (\ref{x_update}). We discuss these considerations in Section \ref{future_directions}.

\subsection{Speed of Output}\label{speed}
To give the reader a sense of the speed at which the $r$-CF generator produces output, we again compare it to a generator that is likely familiar to the reader, the Mersenne Twister, which is widely used for its balance of output speed and statistical quality in non-cryptographic applications. Although the statistical quality of the $r$-CF generator appears to outstrip that of the Mersenne Twister, its speed is markedly slower. We performed a benchmark comparison of the $r$-CF generator to the C++ implementation of the 32-bit Mersenne Twister, {\tt std::mt19937}, included in the {\tt <random>} library. This comparison utilized the Google Benchmark C++ library \cite{google} to measure wall-clock execution times, reported as the average number of nanoseconds to complete an iteration over some number of iterations determined by the user. The benchmark code containing both generators was compiled optimally using the {\tt -O3} and {\tt -march-native} flags, the latter of which is justified since both generators were tested on the same machine within a single run of the software. We collected $50$ estimates of execution times for both generators, each of which was computed over $2 \times 10^8$ iterations to minimize variation between these estimates. To this same end, we ran the code on a single processor. To run the code with these requirements, we used the command {\tt taskset -c 0  ./rcf\_bench --benchmark\_min\_time=200000000x --benchmark\_repetitions=50} in the terminal window. Again to minimize variation, power-saving settings were disabled when executing the code. We also note that we included a ``warm-up'' loop in the benchmark code which was executed before measuring speeds of the $r$-CF and Mersenne Twister generators. Finally, in order to obtain the most consistent and replicable results, these data were gathered when the machine was supplied with power from a wall outlet. The summary statistics of these data are reported in Table \ref{benchmark_table}.

\begin{table}
\centering
\begin{tabular}{lllll}
Generator & Output Type & Sample Size & Mean (ns) & Standard Deviation (ns) \\
\hline
$r$-CF ($n=1000$) & floating-point ({\tt double})& 50& 5.169 & 0.0209 \\
$r$-CF ($n=1000$) & $32$-bit integer ({\tt uint32\_t}) &50 & 6.1364 & 0.0174 \\
$r$-CF ($n=10000$) & $32$-bit integer ({\tt uint32\_t})&50&6.203 & 0.0313 \\
$r$-CF ($n=2$) & $32$-bit integer ({\tt uint32\_t})& 50& 16.528	& 0.1161 \\
Mersenne Twister & 32-bit integer ({\tt uint32\_t}) & 50 & 2.3846 & 0.0468 \\
\end{tabular}
\caption{The mean and standard deviation of $50$ wall-clock execution times of the Mersenne Twister generator and several variations of the $r$-CF generator, measured as the average number of nanoseconds to complete one iteration over $2 \times 10^8$ iterations, estimated by the Google Benchmark C++ library.}
\label{benchmark_table}
\end{table}

Using the wall-clock time statistics in Table \ref{benchmark_table}, we estimate that, on average, the $32$-bit Mersenne Twister produces output at approximately $2.6$ times the rate of the $32$-bit-integer $r$-CF generator, and $2.2$ times the rate of the floating-point $r$-CF generator, as they are presently implemented.

Since the $r$-CF generator algorithm is arithmetically intensive, and does not use bit-level operations, these results are not unexpected. Streamlining the implementation of the $r$-CF generator is yet another avenue of inquiry.

\subsection{Choice of Seed}\label{seed}

Not surprisingly, seeding the generator with a constant vector, $(x_0,x_0,\ldots,x_0)$, where $x_0$ is chosen pseudorandomly from $(0,1)$, results in failure in PractRand. On the other hand, seeding the generator with a ``nearly'' constant vector, that is, setting $(x_0^{(0)}\ldots,x_n^{(0)})$ instead to $(x_0+\varepsilon_0,x_0+\varepsilon_1,\ldots,x_0+\varepsilon_n)$, where each $\varepsilon_j$ is chosen uniformly from an interval of the form $[-10^{-k},10^{-k}]$ for each $0 \leq j \leq n$, the output still performs well in PractRand so long as $k<10$. Beyond this value, the output will begin to show signs of failure, often resulting in PractRand terminating the run when true failure is detected. If, however, we allow the generator to run further by using the {\tt -tlmin} flag, any catastrophic anomalies seem to correct themselves. For example, setting $k=15$, the generator fails after the first iteration of PractRand. If we allow the test to continue, forcing PractRand to examine at least $512$ megabytes to one gigabyte of output, no failures are reported, although some anomalies can persist for some time before resolving themselves.

In summary, even with a deliberately poor choice of seed, perturbing the system by even a small amount allows the system to return relatively quickly to producing what appears to be robust output. 
In this way, and given its performance in repeated testing, it appears that seeding the generator is a relatively straightforward task.

\subsection{Estimating the Period}

The state space of the $r$-CF generator can be as large as the user sees fit. The state of the generator at step $t$ is the vector $(x_0^{(t)},x_1^{(t)},\ldots,x_{n-1}^{(t)})$. If each $x_j^{(t)}$ occupies $b$ bits of memory, then each state occupies $nb$ bits of memory, indicating a maximal state space size of $2^{nb}$. Declaring each $x_j^{(t)}$ as a {\tt double} type in C++, as we have done in our simulations ($53$ available bits), and a vector length of $n=1000$, indicates a very large state space and upper bound on the period. Of course, we are far more interested in estimating a lower bound on the period of a typical run. Given that predicting the period of a finite-precision implementation of the Gauss map after transient effects is not yet generally possible \cite{corless}, predicting the period of the output from Equation (\ref{x_update}) presently appears to be analytically intractable. That said, given the larger state space together with the generator's consistently solid performance in all three test suites used by this effort, we suspect the period to be quite large.  This is a topic which requires further study.

\section{Summary of Effort, Concluding Remarks, and Future Work}\label{future_directions}

The purpose of this effort was to highlight the potential suitability of the $r$-CF map (a generalization of the Gauss continued-fraction map) as a novel starting point for generating  quality pseudorandom output. Although this effort has demonstrated that the map shows much promise in this regard, this paper leaves many stones unturned.

One of the most important considerations is the ultimately periodic behavior of the output, which was not satisfactorily addressed by this effort. Although it is not entirely unreasonable to imagine that the period could be quite long given its performance in testing, this is clearly no substitute for establishing some sort of analytically-obtained lower bound, and future efforts should seek to address this.

Another vital consideration is the speed at which the $r$-CF generator produces output. Using a popular PRNG for comparison, we observed that, on average, the 32-bit Mersenne Twister produces output at approximately $2.6$ times the rate of the $r$-CF generator. This may render the generator unsuitable for applications for which speed is a priority. Future efforts might attempt to understand whether the implementation described in Section \ref{implementation} can be optimized to improve its performance in this regard.

The above trade off in output speed might be justified if the $r$-CF generator, or some variation of it, has the potential to produce cryptographically secure output. As noted earlier, chaotic maps which have received more attention have been modified successfully for this purpose \cite{ahmad,irfan,naik,wang}. Thus, another natural follow-up effort is to investigate the potential for the methods presented here to be used or adapted to create a cryptographically secure PRNG. We note that one of the difficulties presented by the finite-precision Gauss map is that its behavior is heavily influenced by how floating-point arithmetic is implemented \cite{corless}. That is, portability across compiler and hardware could be an issue to consider for such an effort.

For larger values of $n$, the output appears to be quite robust against an obviously poor choice of seed when perturbed by even a small amount, as seen in Section \ref{seed}. Also, given that every run of Dieharder, Practrand, and BigCrush used a different, pseudorandomly chosen seed, it appears that seeding the $r$-CF generator is not at all a difficult task. Still, we must ask if there are any considerations or best practices when seeding this generator, and we leave these investigations for further research.

The $r$-CF generator algorithm naturally suggests an abundance of variations, several of which which we briefly discuss here. To some extent, all of these involve modifications to Equation (\ref{x_update_0}), which updates $r_j$ to simulate an external, ``independent'' source of variation as discussed in Section \ref{rcfprng}. One family of modifications we considered in the course of this effort are those of the form $r_j:=A+(B-A)x_{\psi(j)}$, where $\psi$ is some map from $\{0,1,2,\ldots,n-1\}$ into itself (not necessarily bijective). For instance, choosing a lag other than one, that is, choosing $\psi(j)=\rho(j+\ell)$, where $1<\ell< n$, has yielded promising results. A modification which we tested in the same manner as described in Section \ref{results_and_discussion} is $\psi(j)=\left \lfloor nx_j \right \rfloor$, and we obtained results similar to those given in previous sections of this article. That is, we can say with much certainty that this version passes Dieharder, BigCrush, and PractRand (to $8$ terabytes output). However, in the interest of simplicity,  we opted to present in detail the most natural and computationally efficient method for future efforts to build upon.

In addition to the above, we present two other variations which are worthy of mention, and have shown promise in preliminary testing. The $r$-CF generator may easily be modified into a randomness extractor. Instead of having the system itself simulate $r_j$, take any uniform pseudorandom integer sequence, $\{v_t\}$, on $[0,M]$, where $M$ is a large natural number, and replace Equation (\ref{x_update_0}) with the update rule  $r_j:=A+(B-A)v_t/M$. This method appears to reliably produce output of respectable quality, even if $\{v_t\}$ is of exceedingly poor quality. To illustrate our point, we chose one of the worst options: using the much maligned RANDU generator \cite{knuth} with seed $v_0=1$, one run of both the forward- and reverse-bit generator passed PractRand up to one terabyte with no failures. The output also  consistently passes SmallCrush. Further testing will need to be carried out. Another, much simpler, variation we considered is also guided by the same reasoning presented in Section \ref{rcfprng}: take two values $x_0$ and $x_1$, drawn (pseudo)randomly from $(0,1)$, and generate the sequence $x_{t+2}=G\left(\frac{x_t}{A+(B-A)x_{t+1}}\right)$. One run of both the forward- and reverse-bit output (converted to $32$-bit integers) passed PractRand up to one terabyte with no failures, and the output consistently passes SmallCrush. Again, further testing is required. On one hand, the simplicity of this method offers the advantage of being far more amenable to rigorous analytical methods. On the other hand, the reduced state space raises more pronounced concerns surrounding periodicity and seeding practices, and any future effort would need to address such issues. Such an undertaking would be worthwhile provided the output stands up well to more extensive testing.

With the above in mind, another area of future work is understanding the properties of the exact $r$-CF map and the sequences produced by the $r$-CF generator. Although the output may pass large batteries of statistical tests, its exact nature is far from being well understood. In particular, rigorous results concerning important features, such as discrepancy, gaps, entropy, etc., as well as topological considerations, including mixing properties, are essential for the complete understanding of the suitability of the $r$-CF generator for deployment.

As already mentioned, this effort also puts on display what little is known about the exact $r$-CF map and its finite-precision implementation. It bears repeating an earlier caveat that simulating the dynamics of the $r$-CF map using finite-precision arithmetic is another area which is worthy of future scrutiny. For the Gauss map, $G$, Corless, Frank, and Monroe \cite{corless} demonstrate that the fixed-precision map may still be used to glean properties and behavior of the exact map through simulation. In particular, it turns out that estimating the Lyapunov exponent from a simulated orbit is a reasonable prospect. Also, these efforts demonstrate that the orbits of the finite-precision map are ``machine close'' to the exact orbits of $G$ in the sense that each element of the simulated orbit does not differ from the exact orbit by more than four times the size of the smallest machine representable number in a given implementation. This is all to say that the finite-precision dynamics are a useful approximation of those of the exact map. We ask if similar results exist for the $r$-CF map.

Returning to the exact $r$-CF map, one notably absent consideration in the literature is the seemingly chaotic properties of $T_r$. While $T_r$ appears to inherit some of the classical properties of $G$, it has yet to be seen whether it does so for all of these properties, and for all possible values of $r$. Although, as seen earlier, with the caveat mentioned in the previous paragraph firmly in mind, it does appear that for any $r\geq 1$ almost every orbit of $T_r$ has a positive Lyapunov exponent.

It is worth highlighting here that another generalization of the Gauss map considered by Beck, Tirnakli, and Tsallis \cite{beck} is chaotic. The map $\displaystyle T_{\alpha}(x)=G(x^{\alpha})$, where $\alpha$ is a positive real number, inherits the chaotic nature of $G$ so long as $\alpha$ lies above a critical threshold, $\alpha_c,$ estimated to be close to the value $0.24148514180881.$ We suspect that a two-parameter generalization $T_{r,\alpha}=G(x^{\alpha}/r)$ also inherits the chaotic properties of $G$ for all $r \geq 1$. This two-parameter map certainly appears to have an almost-everywhere positive Lyapunov exponent as seen in Figure \ref{lyap_alpha}. For completeness, we note that in the course of our investigations, we also observed promising statistical results when implementing Equation (\ref{x_update}), except using $T_{r,\alpha}$ instead of $T_r$, and updating $\alpha_{\rho(t)}^{(t)}$ as $\alpha_{\rho(t)}^{(t+1)}:=C+(D-C)x_{\rho(t+2)}^{(t)}$ with $C=0.5$ and $D=1$. However, the computational complexity incurred by computing a non-integer exponent was not worth what appeared to be little, if any, gain in the quality of the output.

\begin{figure*}
\centering
 \begin{tabular}{cc}
 \includegraphics[scale=0.6]{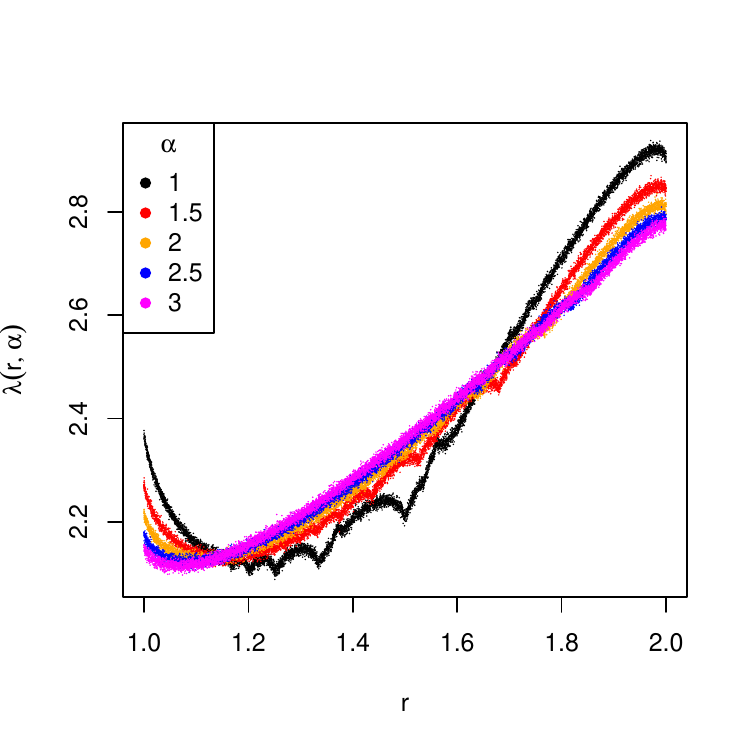} & \includegraphics[scale=0.6]{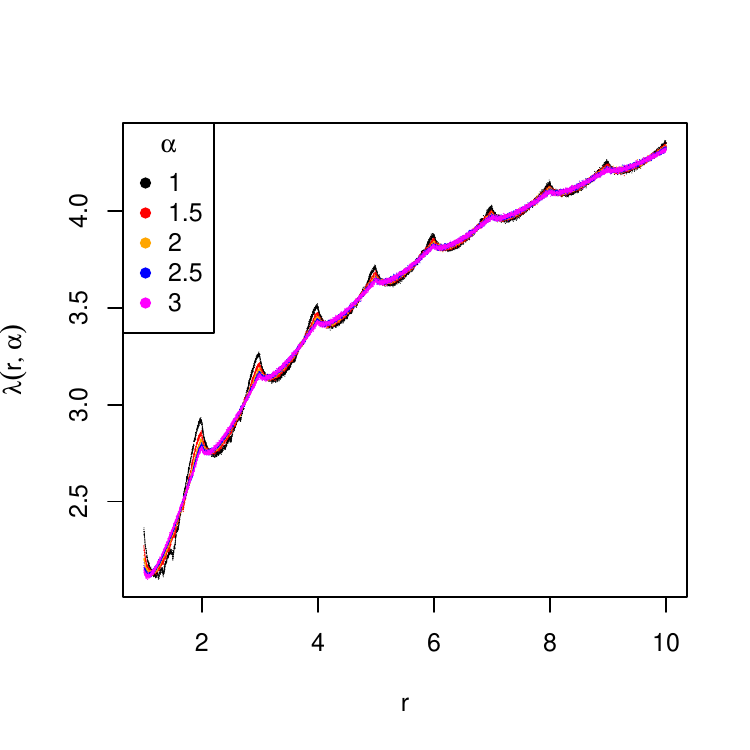} \\
 (a) & (b)
\end{tabular}
\caption[figure]{The estimated Lyapunov exponent, $\lambda$, of $T_{r,\alpha}$ as $r$ varies over $[1,2]$ (a) and $[1,10]$ (b) for $\alpha\in\{1, 1.5,2,2.5,3\}$. Each curve corresponding to a particular value of $\alpha$ was generated by pseudorandomly sampling $10000$ values of $r$ from $[1,2]$ (a) and $[1,10]$ (b), and then for each value of $r$, estimating $\lambda$ from an orbit of a pseudorandomly chosen point in $(0,1)$.}\label{lyap_alpha}
\end{figure*}

We note that in generating Figure \ref{lyap_alpha}, large values of $r$ divided by small values of $x$ raised to a power $\alpha>1$, can result in arithmetic which exceeds the capabilities of finite-precision arithmetic. As $r$ and $\alpha$ increase, such cases become more likely. In these cases, we simply sampled another point value of $r$ and $x_0$ until the desired number of points was achieved. In the case of the $r$-CF generator, as mentioned earlier, when the arithmetic exceeded the capacity of the machine, the generator was seeded with a call to {\tt rand()}. Again, this explains why we chose the parameters $A=1000$ and $B=10000$: higher values tend to result in overflow, requiring more re-seeding calls to {\tt rand()}, resulting in diminished quality and speed.

Finally, Figures \ref{lyap} and \ref{lyap_alpha} suggest the following conjecture:

\begin{conjecture}
Let $\lambda(r,\alpha)$ be the Lyapunov exponent of the orbit of $x_0\in[0,1]$ under $T_{r,\alpha}$. Then, for almost every $x_0$,
$
\lim_{\alpha \rightarrow \infty}\left \Vert 2+\log r -\lambda(r,\alpha)\right\Vert_{\infty}=0.
$
\end{conjecture}

\begin{acknowledgement}
 The author thanks Christopher J. Dugaw of California State Polytechnic University, Humboldt, for reading a preliminary draft of the manuscript and generously offering his time and valued suggestions for improving it.
\end{acknowledgement}

\bibliographystyle{abbrvnat}
\bibliography{rcf}

\end{document}